\documentclass[12pt]{article}
\usepackage[psamsfonts]{amssymb}
\usepackage{amsmath}
\usepackage{amsfonts,euscript,amssymb,amsmath,amsthm,mathrsfs}
\usepackage[english]{babel}
\usepackage{blindtext}

\newcommand{\bbQ}{\mathbb{Q}}

\newcommand{\Image}{\mathop{\rm Im}}

\newcommand{\Real}{\mathop{\rm Re}}

\newenvironment{pfl1}{\par\noindent\textbf{Proof of Lemma 1.}}{\hfill$\square$}

\newtheorem*{theorem}{Theorem}%th
\newtheorem{lemma}{Lemma}%l

\begin{document}
2010 Mathematics Subject Classification 11M41, 11M26

\begin{center}
%\maketitle
{\bf \large On the zeros of linear combinations of degree two L-functions on the critical line. Selberg's approach.}
\end{center}  
\begin{center}
{\bf I.S.~Rezvyakova} 
%\footnote{This work is supported by the RNF grant (project~14-50-00005). } 

{\it Steklov Mathematical Institute of RAS, Moscow}
\end{center}

\begin{flushright}
{\small \it Dedicated to the memory of Professor Atle Selberg \\
and Professor Anatoly Karatsuba} 
\end{flushright}

\begin{abstract}

We consider in details the method of A.~Selberg which allows  one to prove under certain natural conditions that a positive proportion of non-trivial zeros of a linear combination of L-functions from Selberg class lie on the critical line. We  provide all the necessary ingredients to prove this result in the case of a linear combination of degree two L-functions attached to automorphic forms.
\end{abstract}

%% \tableofcontents %% Just for papers exceeding 50 pages.

\begin{center}
{\bf  \Large 1. Introduction}
\label{sec:intro}
\end{center}

In the celebrated work \cite{Selberg1942}, A.~Selberg proved that a positive proportion of non-trivial zeros of the Riemann zeta-function lie on the critical line. Similar result also holds true for Dirichlet L-functions. 

In \cite{Selberg_Amalfi} Selberg introduced a class of L-series (S class, or Selberg class) which are supposed to satisfy an analogue of the Riemann hypothesis. One of the main characteristics of an element from Selberg class is its degree. There is a conjecture which states that a degree $d$ of a function from S class is a non-negative integer. It is known \cite{Richert}, \cite{KP} that there are no elements in Selberg class for non-integer degree $d < 2$. Selberg class of degree 1 is completely described: it consists of the Riemann zeta-function and the imaginary shifts of Dirichlet L-functions. But the structure of Selberg class of degree 2 (as well as of higher degrees) is not clear however. 

The first analogue of Selberg theorem \cite{Selberg1942} for a function of degree two was proved by J.L.~Hafner.  In \cite{Hafner1} he established it for L-functions attached to holomorphic cusp forms of an even weight for the full modular group and trivial character which are also eigenfunctions  of all the Hecke operators, and later in \cite{Hafner2} he established it for L-functions attached to Maass wave forms (for the necessary definitions see, e.g., \cite{Iwaniec_book}). In \cite{Rezvyakova_MZ} we extended a class of L-functions for which an analogue of Selberg theorem is valid and proved it for L-functions attached to automorphic eigen cusp forms of any integer weight $k \ge 1$ for the Hecke congruence group and any character.  

A general linear combination of L-functions from Selberg class (that is the one having no Euler-product decomposition in the region of absolute convergence of the corresponding series) has many zeros outside the critical line and, thereby, violates the Riemann hypothesis (see the works of H.~Davenport and H.~Heilbronn \cite{Dav_Heil},  S.M.Voronin \cite{Voronin}, Y.Lee \cite{Lee} and A.~Selberg \cite{Selberg_Amalfi}). Nevertheless, there is a conjecture stating that almost all non-trivial zeros of such linear combination (satisfying in addition a functional equation of the Riemann type) lie on the critical line. This conjecture was conditionally proved for the Epstein zeta-function by E.~Bombieri and D.~Hejhal \cite{Bom_Hejhal} 
using the GRH and the pair correlation conjecture for the Hecke L-functions. 

The first unconditional result  (and stronger than that could be obtained by means of the Hardy-Littlewood method) on the number of critical line zeros of a linear combination of degree one L-functions was published in 1980 by S.M.~Voronin \cite{Voronin1980}. Later, in 1989, A.A.~Karatsuba \cite{Karatsuba1}, \cite{Karatsuba3} made a substantial improvement in this problem and, finally, in the late 90s, A.~Selberg \cite{Selberg1998}, \cite{Selberg1999} proved that a linear combination (that also obeys an appropriate functional equation) of degree one L-functions has a positive proportion of its non-trivial zeros on the critical line.  Atle Selberg in his proof of this result used besides the methods of \cite{Selberg1942} an additional information about the value distribution  of L-functions on the critical line. 
%An additional ingredient to the method of \cite{Selberg1942} in Selberg's proof of this result for a linear combination of L-functions is an information about the value distribution  of L-functions on the critical line. 
Selberg noted that his method could also work for a linear combination of L-functions of degree two. 
In this work we provide a detailed proof which establishes the result of A.~Selberg for a linear combination of Hecke L-functions. Moreover, this method could be applied for a general linear combination of L-functions from Selberg class.
%for a general linear combination of L-functions from Selberg class. 
Roughly speaking, if we can show for a set of L-functions from Selberg class that each of them has a positive proportion of its non-trivial zeros on the critical line, 
then we will likely be able to prove an analogue of Selberg theorem 
for a linear combination (satisfying a functional equation) of those L-functions. This is achieved by a ``unification'' of the proofs of two results: the analogue of Selberg theorem \cite{Selberg1942} and Selberg density theorem for each L-function. This is then enough  to follow the same arguments for any linear combination to conclude the desired statement. Therefore, we answer the question posted in \cite{Bom_Ghosh} about a rigorous proof of Selberg's result for the Davenport-Heilbronn function but in a more general situation.

The scheme of the proof 
was introduced by Selberg \cite{Selberg1999} and is described in \S 3. 
The proof relies on the estimates (\ref{eq5})--(\ref{eq7}) for 
each L-function from a linear combination as well as on the value distribution result (\ref{eq8}), which we have to establish in every certain case.  All the other details in the proof are similar for any general linear combination. The estimates (\ref{eq5})--(\ref{eq7}) play a crucial role 
in the proof of an analogue of Selberg theorem \cite{Selberg1942} for L-function. We recall (see lemma 4 of \cite{Hafner1} or lemma 2 in \cite{Rezvyakova_MZ}), that both 
(\ref{eq5}) and (\ref{eq6}) follow from  
the same mean-value theorem. Also, by a special choice of the ``mollifier'' presented in our proof, the estimate (\ref{eq7}) as well as Selberg density theorem (\ref{eq14})  (which we shall exploit to establish the value distribution result (\ref{eq8})) are consequences of (\ref{eq6}) (see \S 4).  If we suppose in addition that L-functions from a given linear combination 
are somewhat independent (for example, we may assume the truth of Selberg's orthogonality conjecture for those L-functions), we can prove (\ref{eq8}) applying (\ref{eq14}) (see \S 6). 

All the details in the 
article are given for a linear combination of  Hecke L-functions with complex ideal class group characters of the imaginary quadratic field $\bbQ(\sqrt{-D})$ (which is a linear combination of L-functions from Selberg class of degree two attached to automorphic cusp forms considered in \cite{Rezvyakova_MZ}). 
Analogously to \cite{Rezvyakova_MZ}, we prove in \S5 the estimates (\ref{eq5}) and (\ref{eq6})  but with another  ``mollifier'' and, therefore, following the exposition of \S3 establish the main result formulated in the theorem.
We can also prove an analogous result
for the Epstein zeta function  corresponding to a binary positive definite quadratic form with integer coefficients (i.e., for a linear combination of Hecke L-functions with complex and real ideal class group characters). The necessary mean-value estimates 
corresponding to L-functions with real Hecke characters (which are attached to non-cusp forms) will be given in additional work for completeness of the proof in this case.

\begin{center}
{\bf  \Large 2. Definitions and statement of the main result}
\label{sec:2}
\end{center}

Let $\mathbb Q(\sqrt{-D})$ be the imaginary quadratic extension
of $\mathbb Q$,  and $-D$ be the fundamental discriminant.  Consider a Hecke character $\psi$ of the ideal class group of $\mathbb Q(\sqrt{-D})$. We shall call a Hecke
character \textit{complex}, if it takes at least one complex
(non-real) value, otherwise, we shall call it a \textit{real} character.

For $\operatorname{Re}s>1$ consider a Hecke
$L$-function
$$
L_{\psi}(s)=\sum_{\mathfrak a\in\mathbb I^*} \frac{\psi(\mathfrak
a)}{N\mathfrak a^s},
$$
where the sum is performed over all non-zero ideals $\mathfrak a$ of the  ring of algebraic integers of $\mathbb Q(\sqrt{-D})$, and $N\mathfrak a$ denotes the norm of $\mathfrak a$. The series $
L_{\psi}(s)$ can be written via the Euler product in the right half-plane $\operatorname{Re}s>1$ as
$$
L_{\psi}(s)=\prod_{\mathfrak p}\biggl(1- \frac{\psi(\mathfrak
p)}{N\mathfrak p^s}\biggr)^{-1},
$$
where the product goes over all prime ideals. Let 
$$
r_\psi(n)=\sum_{N\mathfrak a=n}\psi(\mathfrak a). 
$$
From the theory of ideals of imaginary quadratic fields we have the following Euler product representation over prime numbers $p$ for $L_{\psi}(s)$:
\begin{align*}
&L_{\psi} (s) = \sum\limits_{n=1}^{+\infty} \frac{r_{\psi} (n)}{n^s} = \prod\limits_{p} \left( 1 - \frac{r_{\psi}(p)}{p^s} + \frac{\chi_{D} (p)}{p^{2s}}\right)^{-1} = \prod\limits_{p} \left( 1 - \frac{\gamma_{p}}{p^s} \right)^{-1} 
\left( 1 - \frac{\delta_{p}}{p^s} \right)^{-1}, 
%\prod\limits_{p} \left( 1 + \frac{r(p)}{p^s} + \frac{r (p^2)}{p^{2s}} +\ldots \right)^{-1/2}.
%&= \prod_{\substack{p:\\ \chi_D(p)=0}} \biggl( 1
%-\frac{\Psi(p)}{p^s} \biggr)^{1/2} \prod_{\substack{p:\\ \chi_D(p)=-1}}
%\biggl( 1 -\frac{1}{p^{2s}} \biggr)^{1/2} \prod_{\substack{p:\\ \chi_D(p)=1}} \biggl( 1
%-\frac{{\Psi}(p)}{p^s} \biggr)^{1/2} \biggl( 1
%-\frac{\overline{\Psi(p)}}{p^s} \biggr)^{1/2}, 
\end{align*}
%where $\Psi^h = 1$. 
where $|\gamma_{p}|, |\delta_{p}| \in\{ 0, 1 \}$.
Therefore, $r_{\psi(n)}$ can be majorized by the coefficients of the series, defined by the following Euler product:
$$
\prod_{p} \biggl( 1
-\frac{1}{p^s} \biggr)^{-2} = \sum\limits_{n=1}^{+\infty} \frac{\tau(n)}{n^s}.
$$
Thus, 
\begin{equation*}
|r_{\psi} (n)| \le \tau(n).
\end{equation*}

The function $L_{\psi}(s)$ has 
a meromorphic continuation to the whole complex plane with the only pole (which is simple) at $s=1$ in the case of the principle character, and without 
singularities for non-principle Hecke characters. 
E. Hecke (\cite{Hecke1917},~\cite{Hecke1926}) established a functional equation for $L_{\psi} (s)$: {\it let
$$
\Lambda_{\psi}(s)=\biggl(\frac{2\pi}{\sqrt
D}\biggr)^{-s}\Gamma(s)L_{\psi}(s),
$$
then
$$
\Lambda_{\psi}(s)=\Lambda_{\psi}(1-s).
$$}
Hecke also established a connection between Hecke L-functions and modular forms. {\it 
%Let 
%$$
%r_\psi(n)=\sum_{N\mathfrak a=n}\psi(\mathfrak a), 
%$$
%and 
For $z\in\mathbb
H=\{z\colon\operatorname{Im}z>0\}$ set
$$
f(z)=\sum_{n=1}^{+\infty}r_\psi(n)e^{2\pi inz}.
$$
If  $\psi$ is a complex Hecke character, then $f(z)$ is a
holomorphic cusp form of weight~1 for $\Gamma_0(D)$ and the
character $\chi_D(\,\cdot\,)=(\frac{-D}\cdot)$.} In other words,
for every matrix $\gamma=\bigl(\begin{smallmatrix}a&b
\\
c&d\end{smallmatrix}\bigr)$ of the group $\Gamma_0(D)$ (which is a
subgroup of $\text{SL}_2(\mathbb Z)$ that consists of all the
matrices~$\gamma$ satisfying the condition $c\equiv0\ (\operatorname{mod}D)$) the
relation
\begin{equation*}
f(\gamma z) = \chi_{D} (\gamma) (cz +d) f(z),
\end{equation*}
is valid with $\chi_{D} (\gamma) = \chi_{D} (d)$ (see \cite{Hecke1926},~\cite{Iwaniec_book}).

In \cite{Rezvyakova_MZ} we proved that a positive proportion of non-trivial zeros of $L_{\psi}(s)$ lie on the critical line. 
In this work we shall prove the following theorem. 
\begin{theorem}
Suppose that $F(s)  = \sum\limits_{j=1}^{m} c_j L_{j}(s)$ 
is a linear combination of $m$ distinct Hecke $L$-functions attached to complex ideal class group characters
%$$
%F(s) = \sum\limits_{j=1}^{m} c_j L_{j}(s)
%$$
formed with the real coefficients $c_j$. Then a positive proportion of non-trivial zeros of $F(s)$ lie on the critical line. Namely, if we denote by $N_0(T)$ the number of zeros of $F(s)$ on the interval $\{ s= 1/2+it, T\le t \le 2T\}$, then for any large positive $T$ we have 
$$
N_0(T) \gg \frac{1}{m} T \log T.
$$
\end{theorem}

\begin{center}
{\bf  \Large 3. Selberg's approach}
\label{sec:3}
\end{center}

The following idea of detecting zeros of odd order of a real-valued function goes back to H.~Bohr and E.~Landau.
Suppose that we have a real-valued function $F(t)$, and that for some $t$ and $H>0$ we can show that 
\begin{equation}\label{eq2}
\int\limits_{t}^{t+H} |F(u)| du > \left|\int\limits_{t}^{t+H} F(u) du \right|.
\end{equation}
Then $F(\cdot)$ changes its sign on $(t, t+H)$ and, thus, has a zero of odd order on that interval. 
It is easy to show, that if the measure of the subset $\{ t \in (T, 2T) \text{ such that } (\ref{eq2}) \text{ holds} \}$ has its maximal order (i.e. $T$), then the number of zeros of $F$ on the interval $(T, 2T)$ is not less than $c T H^{-1}$ for some absolute positive constant $c$. 

In \cite{Selberg1942} Selberg introduced a special ``mollifier'' to prove his brilliant result on the number of  zeros of the Riemann zeta-function on the intervals of the critical line. For a given $L$-function a ``mollifier'' is connected to $L^{-1} (s)$ and we take it as follows (see \cite{Selberg1999} for the same choice). 
%might be obtained  in the following way. 
For $L_j (s) := L_{\psi_j}(s)$ 
define $\alpha_{j} (\nu)$ by the relation
\begin{equation*}
\sum\limits_{\nu =1}^{+\infty} \alpha_{j}(\nu) \nu^{-s} 
= L_j(s)^{-1/2}.
\end{equation*}
For $X\ge 3$ define a ``mollifier''  $\eta_j (s)$ as the following Dirichlet polynomial 
$$
\eta_j (s) = \sum\limits_{\nu < \sqrt{X}} \frac{\alpha_j (\nu)}
{\nu^{s}} +\sum\limits_{\sqrt{X} \le \nu \le X }  \frac{\alpha_j (\nu)}
{\nu^{s}} \left( 2 - 2 \frac{\log \nu}{\log X} \right)  := \sum\limits_{\nu \le X} \frac{\beta_j (\nu)}{\nu^s},
$$
where 
\begin{equation}
\label{eq4}
\beta_j (\nu ) = \alpha_j (\nu ) \EuScript{L}(\nu)
\quad \text{for} \quad \nu \le X,
\end{equation}
and 
\begin{equation*}
\EuScript{L}(\nu) = \begin{cases} 1 &\text{for }\ \nu < \sqrt{X},
\\
2\dfrac{\log X/\nu}{\log X} & \text{for }\ \sqrt{X} \le \nu \le X, \\
0 &\text{otherwise}.
\end{cases}
\end{equation*}
%The statement that a positive proportion of non-trivial zeros of $L_j(s)$ lie on the critical line is essentially equivalent to the following mean-value estimate. If 
We take $H \asymp \frac{1}{\log T}$ and $X$ as a small (but fixed) power of $T$ (i.e., $\log X \asymp \log T$). Define
\begin{equation*}
I_j(t, H)  = \int\limits_{t}^{t+H} 
\Lambda_j\left( \frac12 + iu\right) \left| \eta_j \left( \frac12 +iu\right)\right|^{2}
\exp\left( \left(\frac{\pi}{2} -\frac{1}{T}\right) u\right) du
\end{equation*}
and suppose that the following estimates hold:
%then
\begin{equation}\label{eq5}
\begin{split}
\int\limits_{T}^{2T} |I_j(t, H)|^2 dt & = O\left( T \frac{H}{\log T}\right), 
\end{split}
\end{equation}
%For Hecke L-functions with complex ideal class group character we proved statement  (\ref{eq5}) with slightly different ``mollifier" in \cite{Rezvyakova_MZ}. The changes in the proof of that statement with the new ``mollifier" will be given in \S~5, as well as 
%the changes in the proof of  the following estimate: 
\begin{equation}\label{eq6}
\int\limits_{T}^{2T} |L_j(1/2+it) \eta_j^2 (1/2+it)|^2 dt = O(T),
\end{equation} 
%which we shall use further. In addition,  
and, setting 
$$
M_j(t, H) = \int\limits_{t}^{t+H} L_j (1/2 +iu) \eta_j^2 (1/2+iu) du - H,
$$ 
that 
%we show in \S~4 the bound 
\begin{equation}\label{eq7}
\begin{split}
\int\limits_{T}^{2T} |M_j(t, H)|^2 dt & = O\left( T \frac{H}{\log T}\right).
\end{split}
\end{equation}
%(without assertion a Hecke character to be a complex) and 
We appeal to the three latter estimates in the proof of the main theorem. 

%Now suppose that we have the bounds (\ref{eq5})--(\ref{eq7}) for L-functions from our linear combination. 
Our next explanation in this section follows closely \cite{Selberg1999} and keeps most of the notation and phrases. To adapt the idea of \cite{Selberg1942} for a linear combination, we need some results on the value distribution of  $\log
|L_j(1/2+it)|$. 
If we have two distinct Hecke L-functions, then 
the difference
$$
\frac{\log|L_{j}(1/2+it)| -  \log|L_{j'}(1/2+it)|}{\sqrt{(n_{j} +n_{j'} )\pi
\log\log t}}
$$
has a normal distribution, where $n_{j}$ equals to $1$ or $2$ depending on 
$\psi_j$ being a complex or a real character, respectively.
More precisely: let 
$\varkappa_{a,b}$ denotes the characteristic function of the
interval
$(a,b)$, then 
\begin{equation}\label{eq8}
\int\limits_{T}^{2T} \varkappa_{a,b} \left( \frac{\log|L_{j}(1/2+it)| - \log|L_{j'}(1/2+it)|}{\sqrt{(n_{j} +n_{j'} )\pi \log\log T}} \right) dt =
T \int\limits_{a}^{b} e^{-\pi u^2} du + O \left( T\varphi(T)\right),
\end{equation}
where $\varphi(T) \to 0$ as $T\to +\infty$. 
%The main ideas for a proof of (\ref{eq8}) are given in \S~6. 
A skeleton of the proof of (\ref{eq8}) (which is similar for different L-functions) is presented in \S~6. 
%The main idea  and rely on the Selberg density theorem 
A very sharp estimate on $\varphi(T)$ can be obtained by means of (\ref{eq12}) 
following the idea of Atle Selberg which was done in details for the Riemann zeta function by K.M.~Tsang in his Ph.D. thesis \cite{Tsang}. In this case one gets 
$$
\varphi(T) \ll \frac{(\log\log\log T)^2}{\sqrt{\log\log T}} .
$$
An easier method of A.~Ghosh \cite{Ghosh} gives a weaker estimate, but for our task we can use any non-trivial result
$$
\varphi(T)  = o(1).
$$
Fix $0< \delta <1/2$ and suppose without loss of generality  that we have (\ref{eq8}) with a monotonically decreasing function $\varphi(T)$ 
such that $\varphi(T)  \gg (\log\log T)^{-\frac12+\delta}$. 
Whence, the subset of the interval $(T, 2T)$ where the inequality 
\begin{equation*}
\left|\log|L_{j}(1/2+it)| -  \log|L_{j'}(1/2+it)|\right| \le
(\log\log T)^{\delta}
\end{equation*}
holds has the measure $O\left(T \varphi(T) \right)$.
Thus, most of the time one of $L_j(1/2+it)$ dominates all the other 
decisively. We show further that this dominance is somewhat persistant over stretches quite long compared to $\dfrac{1}{\log T}$. Suppose $1 \le H \log T \le \varphi^{-1} (T)$. Define 
$$
\Delta_{j} (t, H) = \frac{1}{H} \int\limits_{t}^{t+H} \log
|L_{j}(1/2+iu)| du.
$$
For $0 \le u \le H$  and any positive integer $k$ one can show  that (for the proof see \S 6)
\begin{equation}\label{eq9}
\int\limits_{T}^{2T} \left( \Delta_{j} (t, H) - \log
|L_{j}(1/2+i(t+u))| \right)^{2k} dt = O\left( T e^{Ak}
\left(k^{4k} + k^k \log^k (H\log T)\right)\right).
\end{equation}
Integrating over $u$, we obtain
$$
\int\limits_{T}^{2T} \int\limits_{0}^{H} \left( \Delta_{j} (t, H)
- \log |L_{j}(1/2+i(t+u))| \right)^{2k} du dt = O\left( TH e^{Ak}
\left(k^{4k} + k^k  (\log \log \log T)^k\right)\right).
$$
If we denote by $W_j(t)$ the subset of $u \in [0, H]$ for which
\begin{equation*}
\left|\Delta_{j} (t, H) - \log |L_{j}(1/2+i(t+u))|\right| >
(\log\log T)^{\frac{\delta}{2}},
\end{equation*}
we find, choosing $k$ so large that $k\delta > 7$, that the measure of this subset can be estimated from above as 
$$
\mu \left(W_j(t)\right) \le \frac{H}{(\log\log T)^3},
$$
except for a subset of $t$ from $(T, 2T)$ of measure $O\left(T(\log\log
T)^{-3} \right)$.

If we exclude from $(T, 2T)$ all $t$ such that 
\begin{equation*}
\left| \log|L_{j}(1/2+it)| -  \log|L_{j'}(1/2+t)| \right| \le
(\log\log T)^{\delta}
\end{equation*}
for some $j \ne j'$, and also exclude all 
$t$ such that $\mu \left(W_j(t)\right) > \frac{H}{(\log\log T)^3}$
for some $j$, we get that the interval $(T, 2T)$ apart from a subset of measure $O(m^2 T \varphi(T))$ can be  divided into 
$m$ sets $S_j$ such that for each $t \in S_j$ we have for $j'\ne j$
\begin{equation*}
\log|L_j(1/2+it)| -  \log|L_{j'}(1/2+it)| > (\log\log
T)^{\delta},
\end{equation*}
and for $u\in H_t := (0, H) \setminus \bigcup\limits_{k=1}^{m} W_k (t)$
\begin{equation*}
\begin{split}
&\log|L_j(1/2+i(t+u))| -  \log|L_{j'}(1/2+i(t+u))| = \left( \log|L_j(1/2+i(t+u))| -
\Delta_{j} (t, H) \right) \\
&-
\left( \log|L_j(1/2+it)| - \Delta_{j} (t, H) \right) 
- \left( \log|L_{j'}(1/2+i(t+u))| - \Delta_{j'} (t, H) \right) \\
&+ \left(
\log|L_{j'}(1/2+it)|  - \Delta_{j'} (t, H) \right) 
+ \left(
\log|L_j(1/2+it)| - \log|L_{j'}(1/2+it)| \right) \\
& > (\log\log T)^{\delta} - 4 \left(\log\log T
\right)^{\frac{\delta}{2}} > \frac12  (\log\log T)^{\delta}.
\end{split}
\end{equation*}
From (\ref{eq6}) we see that
$$
\int\limits_{t}^{t+H} |L_j(1/2+iu) \eta_j^2 (1/2+iu)|^2 dt < H
\log\log T
$$
except for a subset of $t\in (T, 2T)$ of measure $O \left(
\dfrac{T}{\log\log T}\right)$. We also exclude those $t$ from $S_j$
without renaming them.

Define $\mathfrak{F}(t)$ by the formula 
$$
\mathfrak{F}(t)= \sum\limits_{j=1}^{m} c_j \Lambda_j\left( \frac12 +
it\right).
$$
Notice, that this is a real-valued function for real $t$, and its real zeros are the zeros of our linear combination on the critical line.

Now, for $t\in S_j$, compare the two values: 
\begin{equation*}
\begin{split}
I(t, H) & = \int\limits_{H_t} \mathfrak{F} (t+u) \left|\eta_j\left(\frac12+i(t+u)\right)\right|^2
\exp\left( \left(\frac{\pi}{2} -\frac{1}{T}\right)
(t+u)\right)du, \\
J(t, H) & = \int\limits_{H_t} \left|\mathfrak{F} (t+u) \eta_j^2\left(\frac12+i(t+u)\right)\right|
\exp\left( \left(\frac{\pi}{2} -\frac{1}{T}\right) (t+u)\right) du.
\end{split}
\end{equation*}
If $J(t, H) > |I(t, H)|$, then $\mathfrak{F}(\cdot)$ changes its sign on $(t,
t+H)$ and so has at least one zero of odd order there.
Using Cauchy's inequality, we find
\begin{equation*}
\begin{split}
I(t, H) & = c_j \int\limits_{u\in H_t} \Lambda_j\left( \frac12+i(t+u)\right)\left|\eta_j\left( \frac12+i(t+u)\right)\right|^2  \exp\left( \left(\frac{\pi}{2} -\frac{1}{T}\right) (t+u)\right) du \\
& + O \left( m \int\limits_{u\in H_t} \left|L_j\left( \frac12+i(t+u)\right)\right|\left|\eta_j\left( \frac12+i(t+u)\right)\right|^2
e^{-\frac12
(\log\log T)^{\delta}}du \right) \\
&= c_j \int\limits_{0}^{H}  \Lambda_j\left( \frac12+i(t+u)\right)\left|\eta_j\left( \frac12+i(t+u)\right)\right|^2 \exp\left(
\left(\frac{\pi}{2} -\frac{1}{T}\right) (t+u)\right) du \\
&+ O\left( \sqrt{\frac{H}{(\log\log T)^3}} \cdot \sqrt{H\log\log
T}\right) + O \left( m \sqrt{\frac{H\cdot H \log\log T}{e^{(\log\log
T)^{\delta}}}}
\right)\\
&= c_j I_j(t, H) +  O\left(
\frac{H}{\log\log T} \right)
\end{split}
\end{equation*}
if $T$ is large enough compared to $m$.
Similarly we get
\begin{equation*}
\begin{split}
J(t, H) &= c_j J_j(t, H) + O\left( \frac{H}{\log\log T} \right),
\end{split}
\end{equation*}
where 
\begin{equation*}
\begin{split}
J_j(t, H) & = \int\limits_{t}^{t+H} |\Lambda_j (1/2+iu)| |\eta_j(1/2+iu)|^2 \exp\left(
\left(\frac{\pi}{2} -\frac{1}{T}\right) u\right) du.
\end{split}
\end{equation*}
For $T \le t \le 2T$, using 
%by means of 
Stirling's formula for the Gamma function we get 
$$
J_j(t, H) \ge e^{-3} \int\limits_{t}^{t+H} |L_j (1/2 +iu) \eta_j^2 (1/2+iu)| du \ge e^{-3} \left( H - |M_j(t, H)|\right).
$$
From (\ref{eq5}) and (\ref{eq7}) it follows, that 
$$
|I_j (t, H)| \le \frac{H}{3 e^3} \quad \text{and} \quad |M_j (t, H)| \le
\frac{H}{3}
$$
outside a set of measure $O\left( \frac{T}{H\log T}\right)$. Therefore, for $t\in S_j$ except for a subset of measure $O\left(
\frac{T}{H\log T}\right)$ the inequality
$$
\frac{|I(t, H)|}{c_j} \le \frac{H}{3 e^3} + O\left( \frac{H}{\log\log T} \right)<
\frac {H}{2 e^3} < \frac{2H}{3 e^3} - O\left( \frac{H}{\log\log T} \right)
\le  \frac{J(t, H)}{c_j},
$$
or
$$
|I(t, H)| < J(t, H)
$$
holds. Taking into account all $j$ we get that $F(\cdot)$ changes its sign on $(t, t+H)$ for $t$ from a subset of $(T, 2T)$ of measure
$$
\sum\limits_{j=1}^{m} \mu(S_j) - O\left( \frac{mT}{H\log
T}\right)= T - O\left( \frac{mT}{H\log T}\right) - O(m^2 T \varphi(T)),
$$
which produces more than $\frac{c}{m} T\log T$ zeros ($c>0$) of $F$ (if $T$ is quite large) taking $H
= \frac{Am}{\log T}$ with large enough constant $A$.

\begin{center}
{\bf  \Large 4. The proof of the mean-value estimate for $M(t,H)$ and Selberg density theorem}
\label{sec:4}
\end{center}
In this section for a given L-function we shall 
%provide the proof of
prove (\ref{eq7}) exploiting the estimate (\ref{eq6}) 
%for any (not necessarily complex) ideal class group character 
(see also \cite{Selberg1999} for the same material). In the end of this section we explain that the same arguments imply Selberg density theorem for this L-function. Throughout the current and the next sections, we shall omit for notational ease the indices $j$ of a given 
%Hecke character and the corresponding 
L-function. 

Using Cauchy's residue theorem, we find 
\begin{equation*}
\begin{split}
 |M(t,H)| &= \left|\int\limits_{t}^{t+H} (L (1/2 +iu) \eta^2 (1/2+iu) -1)  du\right| \\
&\le \left|\int\limits_{1/2}^{3/2} (L (\sigma +it) \eta^2 (\sigma
+it) -1) d\sigma\right| + 
\left|\int\limits_{1/2}^{3/2} \left( L \left(\sigma +i(t+H)\right) \eta^2 \left(\sigma +i(t+H)\right) -1 \right) d\sigma\right| \\
&+ \left|\int\limits_{t}^{t+H} (L (3/2 +iu) \eta^2 (3/2+iu) -1)  du\right|. 
\end{split}
\end{equation*}
The last integral can be estimated as $O(X^{-a/4})$ with $a$ being arbitrary positive number less than $1$, since, by the definition (\ref{eq4}) of $\beta(\nu)$, 
the integrand is the following Dirichlet series 
\begin{equation*}
\begin{split}
\sum\limits_{n>\sqrt{X}}  n^{-3/2-iu} (\sum\limits_{m \nu_1 \nu_2 = n} r(m) \beta(\nu_1) \beta(\nu_2))  
\end{split}
\end{equation*}
with the following bound on its coefficients: 
 \begin{equation}
\begin{split}
|r(m)|\le \tau(m), |\beta(\nu)| \le |\alpha(\nu)|, \sum\limits_{\nu_1 \nu_2 = d} |\alpha(\nu_1) \alpha (\nu_2)| \le \tau(d), \\
|\sum\limits_{m \nu_1 \nu_2 = n} r(m) \beta(\nu_1) \beta(\nu_2)| \le \sum\limits_{m \nu_1 \nu_2 = n} \tau(m) \tau(\nu_1) \tau(\nu_2) = \tau_6(n). \label{eq22}
\end{split}
\end{equation}

Also,
\begin{equation*}
\begin{split}
& \left|\int\limits_{1/2}^{3/2} (L (\sigma +it) \eta^2 (\sigma +it)
-1) d\sigma\right|^2 \\
&\le \int\limits_{1/2}^{3/2} X^{\frac{a}{2} (\frac12-\sigma)}d\sigma
\cdot \int\limits_{1/2}^{3/2} X^{\frac{a}{2} (\sigma-\frac12)} \left|L
(\sigma +it) \eta^2 (\sigma +it) -1\right|^2 d\sigma.
\end{split}
\end{equation*}
If we define 
$$
J_{\sigma} = \int\limits_{T}^{2T} \left|L(\sigma +it) \eta^2
(\sigma +it) -1\right|^2 dt,
$$
we find from (\ref{eq6}) that 
$$
J_{1/2} \ll T + \int\limits_{T}^{2T} \left| L \eta^2 \left( \frac12+it\right)\right|^2 dt
\ll T
%, \quad J_{3/2} = O(T X^{-a})
$$
and from (\ref{eq22}) that  
\begin{equation*}
\begin{split}
 J_{3/2}& = 
\int\limits_{T}^{2T} \left| \sum\limits_{n>\sqrt{X}}  n^{-3/2-it} \left(\sum\limits_{m \nu_1 \nu_2 = n} r(m) \beta(\nu_1) \beta(\nu_2)\right)\right|^2 dt \\
&\ll  T \sum\limits_{n>\sqrt{X}}  \frac{\tau_6^2 (n)}{n^{3}}  + O(1) \ll T X^{-a}.
\end{split}
\end{equation*} 
Now,  an application of Gabriel's convexity theorem yields 
\begin{equation}\label{eq23}
 J_{\sigma} = O\left( T X^{-a
(\sigma- \frac12)}\right).
\end{equation}
Thus, setting, for example, $a=\frac12$, and suggesting that $
H \gg \frac{\log T}{\log^2 X}$ 
%with large enough constant $A>0$,
we obtain 
$$
\int\limits_{T}^{2T} |M(t, H)|^2 dt \ll T X^{-a/2} + T \left( \int\limits_{1/2}^{3/2} X^{-\frac{a}{2} (\sigma-\frac12)}d\sigma\right)^2 = O\left(  T/\log^2 X \right) = O\left(  T H/\log T \right).
$$
%suggesting that 
%$$
%H \gg \frac{\log T}{\log^2 X}.
%$$
%Since $X$ is a small  fixed power of $T$, we can guarantee the truth of the last inequality taking $H = \dfrac{A}{\log T}$ with large enough constant $A>0$. 

The estimate (\ref{eq23}) obtained above is also a sufficient instrument to derive the following Selberg density theorem for a given L-function: 
\begin{equation}\label{eq14}
N(\sigma, T) \ll T^{1-a_1 (\sigma-1/2)} \log T,
\end{equation}
where $a_1$ is some positive constant and $N(\sigma, T)$ is the number of zeros of $L(s)$ in the region $\Real s \ge \sigma, \; T \le \Image s \le 2T$.  From (\ref{eq23}) the density theorem can be established by the standard arguments (see, for example, \cite{Selberg1946} or \cite{Luo}).  We shall exploit the density theorem (\ref{eq14}) 
%this density theorem 
in order to prove the mean-value estimates (\ref{eq12}) which are necessary to establish the limit theorem (\ref{eq8}).
%value-distribution theorem. 

\begin{center}
{\bf  \Large 5. Estimation of Selberg sums for the proof of (\ref{eq5}) and (\ref{eq6}) in case of a complex Hecke character $\psi$}
\end{center}
\label{sec:5}

It is known (see \cite{Hafner1} or  \cite{Rezvyakova_MZ}) that (\ref{eq5}) and (\ref{eq6}) are equivalent to  proper estimates of ``diagonal'' and ``non-diagonal'' terms, which, in turn, are reduced to a certain bound of  Selberg sum (see lemma \ref{l2} below) and to an additive-type problem with the coefficients of $L(s) = L_{\psi} (s)$. The last problem (in case of a complex Hecke character) is overcome, for example, in \cite{Rezvyakova_MZ} (lemma 4) and the proof is not sensitive to a different choice of the coefficients $\beta(\nu)$. 
Therefore, to justify (\ref{eq5}) and (\ref{eq6}) 
%we are then left 
it is sufficient to prove the statement formulated in the next lemma. It is similar to the assertion of lemma 3 in \cite{Rezvyakova_MZ} but with  different values of $\beta(\nu)$ (compare our  notation (\ref{eq4}) 
and the notation (2.2) of  \cite{Rezvyakova_MZ}). 
%and is formulated in the next lemma.
\begin{lemma}
\label{l2} Let $0\le\theta\le\frac14$. Define $S(\theta)$ by
$$
S(\theta)=\sum_{\nu_1,\dots,\nu_4\le X}
\frac{\beta(\nu_1)\beta(\nu_2)\beta(\nu_3)\beta(\nu_4)}{\nu_2
\nu_4}\biggl(\frac q{\nu_1\nu_3}\biggr)^{1-\theta}K\biggl(
\frac{\nu_1\nu_4}q,1-\theta\biggr)K\biggl(\frac{\nu_2
\nu_3}q,1-\theta\biggr),
$$
where $q=(\nu_1\nu_4,\nu_2\nu_3)$ and
\begin{align*}
&K(m,s) = \prod_{p| m} \biggl( 1+ \frac{r^{2}(p)}{p^{s}} +
\frac{r^{2}(p^2)}{p^{2s}}+\dotsb\biggr)^{-1}
\\
&\qquad\qquad\qquad\qquad\qquad\qquad\times\prod_{p^{\alpha} \| m}
\biggl( r(p^{\alpha}) + \frac{r(p^{\alpha+1})r(p)}{p^{s}} +
\frac{r(p^{\alpha+2})r(p^2)}{p^{2s}}+\dotsb\biggr).
\end{align*}
Then the estimate
$$
S(\theta)\ll\frac{X^{2\theta}}{\log X}
$$
holds uniformly in~$\theta$.
\end{lemma}

First, we shall prove the following lemma.
\begin{lemma}\label{l22}
Let $X \ge 3$,
$$
S_{\theta} (X, \gamma, N) = \sum_{\substack{ \nu\le X
\\
(\nu, N)=1}} \frac{\alpha(\nu) K(\nu,
1-\theta)}{\nu^{1-\gamma}} \log \frac{X}{\nu}.
$$
For $0 \le \theta, \gamma \le 1/4$ the estimate  
$$
S_{\theta} (X, \gamma, N) \ll X^{\gamma} \sqrt{\log X}
\prod_{p|N} \biggl( 1+\frac{1}{p} \biggr)^2
$$
holds true.
\end{lemma}

\begin{proof}
The proof is carried out 
%outlined 
essentially the same way as the proof of lemma 3 in \cite{Rezvyakova_MZ}. For $\Real s
>1$, consider the generating function
\begin{equation*}
\begin{split}
h_{\theta, N} (s) &= \sum\limits_{(n, N)=1} \frac{\alpha(n) K(n,
1-\theta)}{n^{s}} = \prod\limits_{p\nmid N} \left( 1 +
\frac{\alpha(p) K(p, 1-\theta)}{p^{s}} + \frac{\alpha(p^2) K(p^2, 1-\theta)}{p^{2s}} +\ldots \right) \\
&= \prod_{p} \left( 1 +
\frac{r^2(p)}{p^{s}}+
\frac{r^2(p^2)}{p^{2s}}+ \ldots\right)^{-1/2}\times \prod\limits_{p\mid N} \left( 1 +
\frac{r^2(p)}{p^{s}}+ \frac{r^2(p^2)}{p^{2s}}+
\ldots\right)^{1/2} \times \\
& \times \prod_{
(p, N)=1} \left( 1 + \frac{r^2(p)}{p^{s}}+
\frac{r^2(p^2)}{p^{2s}}+ \ldots\right)^{1/2} \times \\
&\times \left( 1 + \frac{\alpha(p) K(p, 1-\theta)}{p^{s}} + \frac{\alpha(p^2) K(p^2, 1-\theta)}{p^{2s}} +\ldots \right) .
\end{split}
\end{equation*}
If $0 \le \theta \le 1/4$, $s = \sigma+it$, the product
\begin{equation*}
\begin{split}
N_1(s,\theta) &=\prod_{
(p, N)=1} \left( 1 + \frac{r^2(p)}{p^{s}}+
\frac{r^2(p^2)}{p^{2s}}+ \ldots\right)^{1/2} \times \\
&\times \left( 1 + \frac{\alpha(p) K(p, 1-\theta)}{p^{s}} + \frac{\alpha(p^2) K(p^2, 1-\theta)}{p^{2s}} +\ldots  \right) \\
&= \prod\limits_{(p, N)=1} \left( 1 +
\frac{r^2(p)}{2 p^{s}} + O\left( \frac{1}{p^{2\sigma}}\right)
\right) \left( 1 -
\frac{r^2(p)}{2 p^{s}} + O\left( \frac{1}{p^{\sigma+1-\theta}}\right) + O\left( \frac{1}{p^{2\sigma}}\right)\right)\\
&= \prod\limits_{ (p, N)=1} \left( 1 + O\left(
\frac{1}{p^{2\sigma}}\right) + O\left(
\frac{1}{p^{\sigma+1-\theta}}\right)\right)
\end{split}
\end{equation*}
defines an analytic function in the half-plane $\Real s
> 1/2$. We then have the following identity for 
the generating function in the region $\Real s \ge 1$
\begin{equation*}
\begin{split}
h_{\theta, N} (s) = (D(s))^{-1/2} N_1(s, \theta) G_{N} (s),
\end{split}
\end{equation*}
where 
$$
D(s) = \sum\limits_{n=1}^{+\infty} \frac{r^2(n)}{n^s},
$$
$$
G_{N} (s) = \prod\limits_{p\mid N} \left( 1 +
\frac{r^2(p)}{p^{s}}+ \frac{r^2(p^2)}{p^{2s}}+
\ldots\right)^{1/2}.
$$

An application of  Perron's summation formula gives
\begin{equation*}
\begin{split}
S_{\theta}(X,\gamma,N) &= \frac{1}{2\pi i}
\int\limits_{1-i\infty}^{1+i\infty} h_{\theta, N} (s+1-\gamma)
\frac{X^{s}}{s^2} ds \\
&= \frac{X^{\gamma}}{2\pi i}\int\limits_{1-i\infty}^{1+i\infty}
\frac{X^{s}}{(s+\gamma)^2}\frac{N_1 (s+1, \theta) G_{N} (s+1)
}{\left( D(s+1) \right)^{1/2}} ds.
\end{split}
\end{equation*}
For $\Real s \ge 1$, the following estimates hold:
$$
N_1(s, \theta) \ll 1 \quad \text{for} \quad 0 \le \theta \le 1/4,
$$
$$
G_{N} (s) \ll \prod_{p|N} \left( 1+\frac{1}{p}\right)^2 := G_N.
$$
Let us move the path of integration from the line $\Real s =1$ to the contour,
constructed by the semicircle $\{|s|=(\log X)^{-1},  \Real s \ge 0
\}$ and by the two rays $\{ s= it, |t| \ge (\log X)^{-1} \}$. In the region $\Real s \ge 1$ we shall use the estimate $|D(s)|^{-1}
\ll |s-1|$. The corresponding integral over the semicircle (which we denote by
$K_1$)  can be estimated as follows
\begin{equation*}
\begin{split}
K_1 \ll X^{\gamma} G_{N} \frac{X^{(\log X)^{-1}} (\log^{-1}
X)^{3/2}}{(\log^{-1} X)^2} \ll X^{\gamma} (\log X)^{1/2}
G_{N}.
\end{split}
\end{equation*}
The contribution of the corresponding integral over the two rays (which we denote by $K_2$) is also of the 
required order: 
$$
K_2 \ll X^{\gamma} G_{N} \int\limits_{(\log  X)^{-1}}^{+\infty}
\frac{t^{1/2}}{t^2} dt \ll X^{\gamma} (\log X)^{1/2} G_{N}. 
$$
Thus, the lemma is proved.

\end{proof}

\begin{pfl1}
Using the M\"obius inversion formula,
\begin{equation*}
f(q) = \sum_{d|q} \sum_{m|d} \mu(m) f\biggl(\frac{d}{m}\biggr),
\end{equation*}
we get
\begin{align*}
q^{1-\theta} K\biggl( \frac{\nu_1 \nu_4}{q}, 1-\theta\biggr)
K\biggl( \frac{\nu_2 \nu_3}{q}, 1-\theta\biggr)&= \sum_{d|q}
\sum_{m|d} \mu(m) \biggl( \frac{d}{m} \biggr)^{1-\theta} \times
\\
&\times K\biggl( \frac{\nu_1 \nu_4 m}{d}, 1-\theta\biggr) K\biggl(
\frac{\nu_2 \nu_3 m}{d}, 1-\theta\biggr).
\end{align*}
Substituting this relation into the expression for~$S(\theta)$, we
find
\begin{align*}
S(\theta) &= \sum_{\nu_1,\dots, \nu_4 \le X} \frac{\beta(\nu_1)
\beta(\nu_2) \beta(\nu_3) \beta(\nu_4)}{ \nu_2 \nu_4
(\nu_1\nu_3)^{1-\theta}} \sum_{d|q} \sum_{m|d} \mu(m) \biggl(
\frac{d}{m} \biggr)^{1-\theta}
\\
&\quad\times K\biggl( \frac{\nu_1 \nu_4 m}{d}, 1-\theta\biggr)
K\biggl( \frac{\nu_2 \nu_3 m}{d}, 1-\theta\biggr) = \sum_{d \le X^2}
\sum_{m|d} \mu(m) \biggl( \frac{d}{m} \biggr)^{1-\theta} g^2 (d,m),
\end{align*}
where
$$
g(d,m)=\sum_{\substack{\nu_1\nu_4\equiv0\ (\operatorname{mod}d)
\\
\nu_j\le X}}\frac{\beta(\nu_1)\beta(\nu_4)}{\nu_1^{1-\theta}
\nu_4}K\biggl(\frac{\nu_1\nu_4m}d,1-\theta\biggr).
$$

Representing $\nu_j$ in the form $\nu_j=\delta_j\nu'_j$, where
$(\nu'_j,d)=1$ and all prime divisors of~$\delta_j$ 
%coincide with
are contained among 
those of~$d$ (that is, $\delta_j | d^{\infty}$), we obtain
\begin{align*}
g(d,m) &= \sum_{\substack{\delta_1 \delta_4 \equiv0 \
(\operatorname{mod}d)
\\
\delta_j | d^{\infty}}} \frac{1}{\delta_1^{1-\theta} \delta_4}
\sum_{\substack{\delta_j \nu_j \le X
\\
(\nu_j, d)=1}} \frac{\beta(\delta_1 \nu_1) \beta(\delta_4
\nu_4)}{\nu_1^{1-\theta} \nu_4} K\biggl( \frac{\delta_1 \delta_4
m}{d} \nu_1 \nu_4, 1-\theta\biggr).
\end{align*}
By the definition of $\beta(\nu)$, it follows that
\begin{align*}
g(d,m) &= \sum_{\substack{\delta_1 \delta_4
\equiv0 \ (\operatorname{mod}d)
\\
\delta_j | d^{\infty}}} \frac{\alpha(\delta_1)
\alpha(\delta_4)}{\delta_1^{1-\theta} \delta_4} K\biggl(
\frac{\delta_1 \delta_4 m}{d}, 1-\theta\biggr)
\\
&\times\sum_{\substack{\nu_j \le X/\delta_j
\\
(\nu_j,d) =1 }} \frac{\alpha(\nu_1) \alpha(\nu_4)}{\nu_1^{1-\theta}
\nu_4} K(\nu_1 \nu_4, 1-\theta) \EuScript{L} (\delta_1 \nu_1)
\EuScript{L}(\delta_4 \nu_4).
\end{align*}
Denote the inner sum in the last equality by $S$. Let 
$$
(\nu_1, \nu_4) = n, \quad \nu_1 = n l_1,  \quad \nu_4 = n l_4,  \quad (l_1, l_4)=1.
$$ 
Then 
\begin{align*}
S &= \sum\limits_{(n,d)=1} \frac{1}{n^{2-\theta}} \sum_{\substack{l_j \le X/\delta_j n
\\
(l_j,d) =1 , \\ (l_1, l_4)=1}} \frac{\alpha(n l_1) \alpha(n l_4)}{l_1^{1-\theta}
l_4} K(n^2 l_1 l_4, 1-\theta) \EuScript{L} (\delta_1 n l_1)
\EuScript{L}(\delta_4 n l_4)\\
&= \sum\limits_{(n,d)=1} \frac{1}{n^{2-\theta}} \sum_{\substack{k_1, k_4 | n^{\infty}, \\ (k_1, k_4)=1}} \frac{\alpha(n k_1) \alpha(n k_4) K(n^2 k_1 k_4, 1-\theta)}{k_1^{1-\theta} k_4}  \\
&\times\sum_{\substack{l_j \le X/\delta_j k_j n
\\
(l_j,nd) =1 , \\ (l_1, l_4)=1}} \frac{\alpha(l_1) \alpha(l_4)}{l_1^{1-\theta}
l_4} K(l_1, 1-\theta) K( l_4, 1-\theta) \EuScript{L} (\delta_1 n k_1 l_1)
\EuScript{L}(\delta_4 n k_4  l_4).
\end{align*}
We now use the following identity:
\begin{align*}
&\sum_{(l_1,l_2)=1} f(l_1,l_2) =\sum_{l_1,l_2} f(l_1,l_2) \sum\limits_{r \mid (l_1, l_2)} \mu(r) = 
\sum\limits_{r} \mu(r) \sum_{l_1, l_2} f(r l_1, r l_2).
\end{align*}
Whence, {\allowdisplaybreaks
\begin{align*}
S&= \sum\limits_{(n,d)=1} \frac{1}{n^{2-\theta}} \sum\limits_{(r, nd)=1}  \frac{\mu(r)}{r^{2-\theta}}\sum_{\substack{k_1, k_4 | n^{\infty}, \\ (k_1, k_4)=1}} \frac{\alpha(n k_1) \alpha(n k_4) K(n^2 k_1 k_4, 1-\theta)}{k_1^{1-\theta} k_4}  \\
&\times\sum_{\substack{l_j \delta_j k_j n r \le X
\\
(l_j,nd) =1}} \frac{\alpha(r l_1) \alpha(r l_4)}{l_1^{1-\theta}
l_4} K(r l_1, 1-\theta) K( r l_4, 1-\theta)  \EuScript{L} (l_1 \delta_1 k_1 n r)
\EuScript{L}( l_4 \delta_4  k_4 n r )\\
&= \sum\limits_{(n,d)=1} \frac{1}{n^{2-\theta}} \sum\limits_{(r, nd)=1}  \frac{\mu(r)}{r^{2-\theta}}\sum_{\substack{k_1, k_4 | n^{\infty}, \\ (k_1, k_4)=1}} \frac{\alpha(n k_1) \alpha(n k_4) K(n^2 k_1 k_4, 1-\theta)}{k_1^{1-\theta} k_4}  \\
&\times \sum_{l_j | r^{\infty}} \frac{\alpha(r l_1) \alpha(r l_4)}{l_1^{1-\theta}
l_4} K(r l_1, 1-\theta) K( r l_4, 1-\theta) \\
&\times \sum_{\substack{\nu_j \delta_j k_j l_j n r \le X
\\
(\nu_j,ndr) =1}} \frac{\alpha(\nu_1) \alpha(\nu_4)}{\nu_1^{1-\theta}
\nu_4} K(\nu_1, 1-\theta) K( \nu_4, 1-\theta)  \EuScript{L} (\nu_1 l_1 \delta_1 k_1 n r)
\EuScript{L}( \nu_4 l_4 \delta_4  k_4 n r ) .
\end{align*}
}
Set 
\begin{align}
\notag S_{\theta}(X, A,\gamma,N) &= \sum_{\substack{\nu\le X/A
\\
(\nu, N)=1}} \frac{\alpha(\nu) K(\nu,
1-\theta)}{\nu^{1-\gamma}} \EuScript{L}(A \nu). 
\end{align}
From the definition of $\EuScript{L}(\nu)$ we obtain 
$$
S_{\theta}(X, A,\gamma,N) = \sum_{\substack{1 \le \nu < \sqrt{X}/A
\\
(\nu, N)=1}} \frac{\alpha(\nu) K(\nu,
1-\theta)}{\nu^{1-\gamma}}  + \sum_{\substack{\sqrt{X}/A \le \nu\le X/A
\\
(\nu, N)=1}} \frac{\alpha(\nu) K(\nu,
1-\theta)}{\nu^{1-\gamma}} 2 \frac{\log \frac{X}{A\nu}}{\log X}.
$$
For $\nu < \sqrt{X}/A$, using the identity
$$
2 \frac{\log \frac{X}{A\nu}}{\log X} - 2 \frac{\log \frac{\sqrt{X}}{A\nu}}{\log X} =1,
$$
we find that 
$$
S_{\theta}(X, A,\gamma,N) = \frac{2}{\log X} \left( S_{\theta}(X/A,\gamma,N)-S_{\theta}(\sqrt{X}/{A},\gamma,N) \right),
$$
where the function $S_{\theta} (X, \gamma, N)$ of three variables is defined in lemma \ref{l22}.

Thus, we finally obtain:
\begin{align*}
g(d, m) &\ll \frac{X^{\theta}}{\log X}  \prod_{p|d} \biggl( 1+\frac{1}{p} \biggr)^4\sum_{\substack{\delta_1 \delta_4 \equiv0 \ (\operatorname{mod}d)
\\
\delta_j | d^{\infty}}}\!\! \frac{|\alpha(\delta_1)
\alpha(\delta_4)|}{\delta_1 \delta_4} K\biggl( \frac{\delta_1
\delta_4 m}{d}, 1-\theta\biggr)  \\
&\le  \frac{X^{\theta}}{\log X}   \frac{1}{d} \prod_{p|d} \biggl( 1+\frac{1}{p} \biggr)^4 \sum_{n| d^{\infty}} \frac{K(nm, 1-\theta)}{n}  \sum_{\substack{\delta_1 \delta_4 = nd
\\
\delta_j | d^{\infty}}}\!\! |\alpha(\delta_1)
\alpha(\delta_4)|.
\end{align*}
Since $\alpha(n)$ is a multiplicative function, the function  
$$
b(n) = \sum\limits_{n_1 n_2 =n} |\alpha
(n_1) \alpha(n_2)|
$$ 
is also multiplicative. 
Thereby, using the estimate 
\begin{equation}\label{eq30}
K(n,1-\theta)\ll\tau (n) \text{ for } 0 \le \theta \le 1/2
\end{equation}
(which follows from the definition of $K(\cdot, 1-\theta)$ and the estimate $|r(n)| \le \tau(n)$), we find that
\begin{align*}
g(d, m) &\ll  \frac{X^{\theta}}{\log X}   \tau (m)  \frac{1}{d} \prod_{p|d} \biggl( 1+\frac{1}{p} \biggr)^4 
\sum_{n | d^{\infty}} \frac{K(n, 1-\theta)}{n}  b(nd) \\
&\ll \frac{X^{\theta}}{\log X}   \tau (m)  \frac{1}{d} \prod_{p^{\beta}||d} \biggl( 1+\frac{1}{p} \biggr)^4 \left( b(p^{\beta}) + \frac{b(p^{\beta+1}) K(p, 1-\theta) }{p} + \frac{b(p^{\beta+2}) K(p^2, 1-\theta) }{p^2} +\ldots\right).
 \end{align*}

Let us estimate the values $b(n)$ from above. From the definition of the coefficients $\alpha(n)$ we have:
\begin{align*}
&\sum\limits_{n=1}^{+\infty} \frac{\alpha(n)}{n^s} = L^{-1/2} (s) = \prod\limits_{p} \left( 1 - \frac{r(p)}{p^s} + \frac{\chi_{D} (p)}{p^{2s}}\right)^{1/2} = \prod\limits_{p} \left( 1 - \frac{\gamma_{p}}{p^s} \right)^{1/2} 
\left( 1 - \frac{\delta_{p}}{p^s} \right)^{1/2}, 
%\prod\limits_{p} \left( 1 + \frac{r(p)}{p^s} + \frac{r (p^2)}{p^{2s}} +\ldots \right)^{-1/2}.
%&= \prod_{\substack{p:\\ \chi_D(p)=0}} \biggl( 1
%-\frac{\Psi(p)}{p^s} \biggr)^{1/2} \prod_{\substack{p:\\ \chi_D(p)=-1}}
%\biggl( 1 -\frac{1}{p^{2s}} \biggr)^{1/2} \prod_{\substack{p:\\ \chi_D(p)=1}} \biggl( 1
%-\frac{{\Psi}(p)}{p^s} \biggr)^{1/2} \biggl( 1
%-\frac{\overline{\Psi(p)}}{p^s} \biggr)^{1/2}, 
\end{align*}
%where $\Psi^h = 1$. 
where $|\gamma_{p}|, |\delta_{p}| \in\{ 0, 1 \}$.
Therefore, the second identity yields the estimate 
$$
b(p) = |r(p)|, 
$$
and the last identity implies, that $b(n)$ can be also majorized by the coefficients of the series, defined by the following Euler product:
$$
\prod_{p} \biggl( 1
-\frac{1}{p^s} \biggr)^{-2} = \sum\limits_{n=1}^{+\infty} \frac{\tau(n)}{n^s}.
$$
Thus, 
\begin{equation*}
b(n) \le \tau(n).
\end{equation*}
We define the multiplicative function $B(n)$ by the relations 
\begin{equation}\label{eq31}
B(p) = |r(p)|, \quad B(p^{\beta}) = \tau(p^{\beta}) \text{ for } \beta >1.
\end{equation}
Hence, from (\ref{eq30}),  (\ref{eq31}) we obtain 
\begin{align*}
S(\theta) &\ll \frac{X^{2\theta}}{\log^2 X}\!\! \sum_{d
\le X^2} \frac{1}{d^{1+\theta}} \prod_{p^{\beta}||d} \biggl(
1+\frac{1}{p}\biggr)^{8} \left( B(p^{\beta}) + \frac{B(p^{\beta+1}) K(p, 1-\theta) }{p} + \frac{B(p^{\beta+2}) K(p^2, 1-\theta) }{p^2} +\ldots \right)^2 \times \\
&\times \sum_{m|d} \frac{\mu^2(m) \tau^2(m)}{m^{1-\theta}}
\ll \frac{X^{2\theta}}{\log^2 X}\sum_{d \le X^2}
\frac{1}{d} \prod_{p^{\beta}||d}  \left( B(p^{\beta}) + \frac{B(p^{\beta+1}) C}{p} \right)^2  \biggl(
1+\frac{1}{p}\biggr)^{8} \biggl( 1+ \frac{\tau^2(p)}{p^{1-\theta}} \biggr) \\
&\ll \frac{X^{2\theta}}{\log^2 X}\sum_{d \le X^2}
\frac{1}{d} \prod_{p^{\beta}||d}  \left( B(p^{\beta}) + \frac{B(p^{\beta+1}) C}{p} \right)^2  \biggl(
1+\frac{1}{\sqrt{p}}\biggr)
\end{align*}
for some positive constant $C$. 

To estimate the last sum over $d$, 
let us write the following chain of formal equalities:
\begin{equation*}
\begin{split}
&\sum_{d=1}^{+\infty}
\frac{1}{d^s} \prod_{p^{\beta}||d}  \left( B(p^{\beta}) + \frac{B(p^{\beta+1}) C}{p} \right)^2  \biggl(
1+\frac{1}{\sqrt{p}}\biggr)  \\
&= \prod_{p}
\left( 1+\frac{1}{p^s} \left( 1+ \frac{1}{\sqrt{p}}\right) \left( |r (p)|+ \frac{\tau(p^2) C}{p}\right)^2 +\frac{1}{p^{2s}} \left( 1+ \frac{1}{\sqrt{p}}\right) \left( \tau (p^2) + \frac{\tau(p^3) C}{p}\right)^2+\ldots \right)\\
&=
%\prod_{p} \left( 1+\frac{r^2(p)}{p^s}\right) \left(
%1+\frac{r^2(p)}{p^{s+\frac12} (1+\frac{r^2(p)}{p^s})}+ \frac{\left( 1+ \frac{1}{\sqrt{p}}\right) \left( 2|r (p)|\tau(p^2) C + \frac{\tau^2(p^2) C^2}{p}\right) }{p^{s+1} (1+\frac{r^2(p)}{p^s})}\right. \\
%&\left. + \frac{\left( 1+ \frac{1}{\sqrt{p}}\right) \left( \tau (p^2) + \frac{\tau(p^3) C}{p}\right)^2}{p^{2s}\left( 1+\frac{r^2(p)}{p^s} \right)}+\ldots \right) 
%= 
\sum\limits_{d=1}^{+\infty} \frac{1}{d^s} \left( \sum\limits_{d_1
d_2 =d} b_1 (d_1) b_2 (d_2)\right),
\end{split}
\end{equation*}
where
\begin{equation*}
\begin{split}
\sum\limits_{n=1}^{+\infty} \frac{b_1 (n)}{n^s} &= \prod_{p}
\left( 1+\frac{r^2(p)}{p^s}\right), \\
\sum\limits_{n=1}^{+\infty} \frac{b_2 (n)}{n^s} &= \prod_{p}
\left(
1+\frac{r^2(p)}{p^{s+\frac12} (1+\frac{r^2(p)}{p^s})}+ \frac{\left( 1+ \frac{1}{\sqrt{p}}\right) \left( 2|r (p)|\tau(p^2) C + \frac{\tau^2(p^2) C^2}{p}\right) }{p^{s+1} (1+\frac{r^2(p)}{p^s})}\right. \\
&\left. + \frac{\left( 1+ \frac{1}{\sqrt{p}}\right) \left( \tau (p^2) + \frac{\tau(p^3) C}{p}\right)^2}{p^{2s}\left( 1+\frac{r^2(p)}{p^s} \right)}+\ldots \right) .
\end{split}
\end{equation*}

Thus, 
$$
S(\theta)\ll \frac{X^{2\theta}}{\log^2 X}\sum_{d_1 d_2 \le X^2}
\frac{b_1 (d_1) b_2 (d_2)}{d_1 d_2}.
$$
We shall estimate the latter sum in the following way:
\begin{equation*}
\begin{split}
&\sum\limits_{d_1 d_2 \le X^2}
\frac{b_1 (d_1) b_2 (d_2)}{d_1 d_2} 
\le  \sum\limits_{n \le X^2}
\frac{r^2(n)}{n}  
\prod_{p} 
\left(
1+\frac{r^2(p)}{p^{\frac32} (1+\frac{r^2(p)}{p})} \right. \\
&\left. +\frac{\left( 1+ \frac{1}{\sqrt{p}}\right) \left( 2|r (p)|\tau(p^2) C + \frac{\tau^2(p^2) C^2}{p}\right) }{p^{2} (1+\frac{r^2(p)}{p})}
+ \frac{\left( 1+ \frac{1}{\sqrt{p}}\right) \left( \tau (p^2) + \frac{\tau(p^3) C}{p}\right)^2}{p^{2}\left( 1+\frac{r^2(p)}{p} \right)}+\ldots \right) \\
& \ll \sum\limits_{n \le X^2}
\frac{r^2(n)}{n} 
\ll \log X
\end{split}
\end{equation*}
in view of the bound 
$$
\sum\limits_{n\le x} |r(n)|^2 \ll  x 
$$
(see \cite{Hecke1937} or \cite{Rankin1939}). 
Thus, 
\begin{align*} 
S(\theta)
&\ll \frac{X^{2\theta}}{\log X}
\end{align*}
and the lemma is proved.

\end{pfl1}

\begin{center}
{\bf \Large 6. The distribution of the values of $L_{\psi}(s)$ on the critical line}
\label{sec:6}
\end{center}

The core of the proof of the value distribution formula (\ref{eq8}) are the following mean estimates:
\begin{equation}\label{eq12}
\int\limits_{T}^{2T} \left| \log |L_{\psi}(1/2+it)| - \Real \sum\limits_{p < z} \frac{r_{\psi}(p)}{p^{1/2+it}}\right|^{2k} dt = O(T k^{4k} e^{Ak})
\end{equation}
where $\frac{\log z}{\log T} \asymp \frac{1}{k}$. 
As in the case of the Riemann zeta-function, to prove the above estimates we need a special asymptotic expansion of $\log L_{\psi}(s)$ on the critical line. One could come to it in few steps. 

First, the Euler product factorization gives the identity 
$$
\log L_{\psi}(s) = \sum\limits_{n=2}^{+\infty} \frac{\Lambda_{\psi} (n)}{\log n} n^{-s}, \quad \Real s >1.
$$ 
Using this formula and contour integration, one can obtain the following relation for $t \ge 2$ ($t \ne \Image \rho$), $2\le x\le t^2$:
\begin{equation*}
\begin{split}
&\log L_{\psi}(1/2 +it) =  \sum\limits_{n \le x^3} \frac{\Lambda_{x,\psi} (n)}{\log n} n^{-\sigma_{x,t} - it} \\
&+  O \left( \left(\sigma_{x,t} - \frac12 \right)  \left( 1+ \log^{+}\frac{1}{\eta_t \log x}  + \left(\sigma_{x,t} - \frac12 \right)  \log x\right) \left(  \left|\sum\limits_{n \le x^3} \frac{\Lambda_{x,\psi} (n)}{\log n} n^{-\sigma_{x,t} - it} \right| + \log t \right) \right),
\end{split}
\end{equation*}
where 
$$
\eta_t = \min_{\rho= \beta + i\gamma} |t-\gamma|,
$$
$$
\sigma_{x,t}  = \frac12 +2\max_{\rho} \left( \beta - \frac12, \frac{2}{\log x} \right)
$$
and maximum is taken over all zeros $\rho = \beta +i\gamma$ of $L_{\psi}(s)$ satisfying the following conditions 
$$
|t-\gamma| \le \frac{x^{3 \left( \beta -\frac12\right)}}{\log x}, \quad \beta \ge \frac12
$$
(we refer the reader to the articles \cite{Selberg1946}, \cite{Tsang}, \cite{Sank} for more details). 
%Finally, it can be derived that  
%\begin{equation}\label{eq13}
%\begin{split}
%\log L_{\psi}\left(\frac12 +it \right) &- \sum\limits_{p < x^3} \frac{r_{\psi} (p)}{p^{\frac12+it}} =
%O \left( \left| \sum\limits_{p < x^3} \frac{\Lambda_{\psi} (p)- \Lambda_{x,\psi} (p)}{\log p} p^{-\frac12 - i%t} \right|\right) \\
%&+ O \left( \left| \sum\limits_{p < x^{3/2}} \frac{\Lambda_{x,\psi} (p^2)}{\log p} p^{-1 - 2it} \right|\right) +  O \left( \left(\sigma_{x,t} - \frac12 \right) \left( 1+ \log^{+}\frac{1}{\eta_t \log x} \right) \log t \right) \\
%&+O \left( \left(\sigma_{x,t} - \frac12 \right)  \left( 1+ \log^{+}\frac{1}{\eta_t \log x} \right) x^{\sigma_{x,t} - \frac12} \int\limits_{1/2}^{+\infty} x^{1/2-\sigma} \left|\sum\limits_{p < x^3}\frac{\Lambda_{x,\psi} (p) \log xp}{p^{\sigma + it}}  \right| d\sigma \right).
%\end{split}
%\end{equation}
Finally, one can prove (\ref{eq12}) for positive integer $k$ using the density theorem  (\ref{eq14}). 
%From here, using a density theorem  (\ref{eq14}), one can prove (\ref{eq12}) for positive integer $k$.
The initial statement (\ref{eq8}) is now a consequence (following the methods of 
\cite{Ghosh} or \cite{Tsang}) of (\ref{eq12}) and the 
asymptotic formula
\begin{equation*}
\begin{split}
\sum\limits_{p \le z} \frac{(r_{\psi} (p)- r_{\psi'} (p))^2}{p} &= \sum\limits_{p \le z} \frac{r^2_{\psi} (p)}{p} + \sum\limits_{p \le z} \frac{r^2_{\psi'} (p)}{p} -\sum\limits_{p \le z} \frac{r_{\psi\cdot \psi'} (p) + r_{\psi\cdot  \overline{\psi'}} (p)}{p} \\
&= (n_{\psi}+ n_{\psi'}) \log\log z + O(1)
\end{split}
\end{equation*}
which is valid for $\psi \ne \psi', \overline{\psi'}$ (which we can suppose, since otherwise $L_{\psi} (s) = L_{\psi'} (s) $).

We shall also explain in this section how to prove the estimate (\ref{eq9}).
% using the approximation (\ref{eq13}). 
%If we denote the RHS of (\ref{eq13}) 
Define by $Err_{\psi} (t)$ the difference
\begin{equation*}
Err_{\psi} (t) = \log L_{\psi}\left(\frac12 +it \right) - \sum\limits_{p < z} \frac{r_{\psi} (p)}{p^{\frac12+it}}. 
\end{equation*}
Its contribution to (\ref{eq9}) 
is as small as in (\ref{eq12}), since
\begin{equation*}
\int\limits_{T}^{2T} \left| \frac{1}{H} \int\limits_{t}^{t+H} Err (u) du \right|^{2k} dt \ll 
 \frac{1}{H^{2k}} \int\limits_{T}^{2T} H^{2k-1} \int\limits_{t}^{t+H} \left| Err (u) \right|^{2k} du   dt\ll \int\limits_{T}^{2T+1} \left| Err (t) \right|^{2k} dt \ll 
T k^{4k} e^{Ak} .
\end{equation*}
Now we consider the contribution to (\ref{eq9}) coming from the main term, namely
\begin{equation*}
\int\limits_{T}^{2T} \left| \sum\limits_{p<z} \frac{r_{\psi}(p)}{p^{1/2+it}} \left( \frac{1}{H} \int\limits_{0}^{H} \frac{du}{p^{iu}} - p^{-iu_1} \right) \right|^{2k} dt
\end{equation*}
 with $0\le u_1 \le H$, which in turn depends on the diagonal part 
\begin{equation*}
\sum\limits_{p<z} \frac{r_{\psi}^2 (p)}{p} \left| \frac{1}{H} \frac{p^{-iH} -1}{-i \log p} - p^{-iu_1} \right|^2.
\end{equation*}
We divide the last sum into two and collect in the first sum $S_1$ the terms with $p< e^{1/H}$ and in the second sum $S_2$ the remaining terms. Thus, 
\begin{equation*}
\begin{split}
S_1 &\ll \sum\limits_{p< e^{1/H}} \frac{r_{\psi}^2 (p)}{p} \left|\frac{-1+ (1-iH\log p - H^2/2\log^2 p)}{-i H\log p} - (1 -iu_1\log p) + O(H^2\log^2 p) \right|^2 \\
&\ll \sum\limits_{p< e^{1/H}} \frac{r_{\psi}^2 (p)}{p} H^2 \log^2 p \ll H^2  \sum\limits_{p< e^{1/H}} \frac{\log^2 p}{p}  \ll H \sum\limits_{p< e^{1/H}} \frac{\log p}{p}  \ll 1.
\end{split}
\end{equation*}
For $S_2$ we have the following estimate
\begin{equation*}
S_2 \ll \sum\limits_{ e^{1/H} \le p <z}  \frac{r_{\psi}^2 (p)}{p} \ll  \sum\limits_{ e^{1/H} \le p <z}  \frac{1}{p} \ll \log (H \log z) \ll \log (H\log T).
\end{equation*}
And, finally, (\ref{eq9}) can be obtained.

\section*{Acknowledgment}

{I would like to thank many people without whom this work may not appear. 
%I wouldn't have made clear for myself the wonderful work \cite{Selberg1999} of Atle Selberg. 
First of all, I am thankful to Michel Balazard and the Russian-French Laboratorie J.-V. Poncelet of Moscow Center for Continuous Mathematical Education for organizing the series of conferences on ``Zeta-functions'', and especially to Eric Saias for sharing with me the notes of Atle Selberg \cite{Selberg1999}, I am thankful to Professor Matti Jutila for giving me the references to the papers solving the convolution-type problems, to Professor Haseo Ki for the opportunity to participate in the conference ``Zeta function days in Seoul'' in 2009 where I met 
several dozens of giants in analytic number theory, to Professor Kai-Man Tsang for sharing the text of his Ph.D thesis and 
for series of conferences on Number theory organized in Hong Kong, 
to Professor Aleksandar Ivi\'c for a great support and to many friends (especially to Leo Murata and Maxim Korolev) for their interest on my research.}

\hspace{5mm}
\begin{flushleft}
{\it Irina Rezvyakova \\
Department of Algebra and Number Theory \\
Steklov Mathematical Institute\\
Moscow, Russia\\
irezvyakova@gmail.com, rezvyakova@mi.ras.ru}
\end{flushleft}
\end{document}